\documentclass[12 pt]{article}
 
\usepackage{amsmath, amsthm, amssymb,enumerate,hyperref,tikz,fullpage}
\usetikzlibrary{calc,shapes}

\newcommand{\class}[1]{{\sc #1}}
\newcommand{\complete}[1]{{\sc #1}-{complete}}

\newcommand{\for}[1]{\ensuremath E^+(#1)}
\newcommand{\bac}[1]{\ensuremath E^-(#1)}

\newcommand{\wt}{\operatorname{wt}}

\tikzset{
  bigblue/.style={circle, draw=blue!80,fill=blue!40,thick, inner sep=1.5pt, minimum size=5mm},
  bigred/.style={circle, draw=red!80,fill=red!40,thick, inner sep=1.5pt, minimum size=5mm},
  bigblack/.style={circle, draw=black!100,fill=black!40,thick, inner sep=1.5pt, minimum size=5mm},
  bluevertex/.style={circle, draw=blue!100,fill=blue!100,thick, inner sep=0pt, minimum size=2mm},
  redvertex/.style={circle, draw=red!100,fill=red!100,thick, inner sep=0pt, minimum size=2mm},
  blackvertex/.style={circle, draw=black!100,fill=black!100,thick, inner sep=0pt, minimum size=2mm},  
  whitevertex/.style={circle, draw=black!100,fill=white!100,thick, inner sep=0pt, minimum size=2mm},  
  smallblack/.style={circle, draw=black!100,fill=black!100,thick, inner sep=0pt, minimum size=1mm},
  smallwhite/.style={circle, draw=black!100,fill=white!100,thick, inner sep=0pt, minimum size=1mm},    
}

\title{A Dichotomy Theorem for Circular Colouring Reconfiguration}

\makeatletter
\renewcommand*\@fnsymbol[1]{\the#1}  
\makeatother

\author{Richard C. Brewster\thanks{Department of Mathematics and Statistics, Thompson Rivers University, Kamloops, Canada. E-mail: \texttt{rbrewster@tru.ca}, \texttt{smcguinness@tru.ca}.}
\and Sean McGuinness\footnotemark[1] 
\and Benjamin Moore\thanks{Department of Mathematics, Simon Fraser University, Burnaby, Canada. E-mail: \texttt{brmoore@sfu.ca}.}
\and Jonathan A. Noel\thanks{Mathematical Institute, University of Oxford, Oxford, United Kingdom. E-mail: \texttt{noel@maths.ox.ac.uk}.} }

\newtheoremstyle{case}{}{}{\normalfont}{}{\itshape}{:}{ }{}

\newtheorem{thm}{Theorem}
\newtheorem{lem}[thm]{Lemma}

\newtheorem{prop}[thm]{Proposition}
\newtheorem{conj}[thm]{Conjecture}
\newtheorem{cor}[thm]{Corollary}

\newtheorem{lemma}[thm]{Lemma}

\theoremstyle{definition}

\newtheorem*{ack}{Acknowledgements}

\newtheorem{rem}[thm]{Remark}

\newtheoremstyle{case}{}{}{\normalfont}{}{\itshape}{\normalfont:}{ }{}

\theoremstyle{case}

\tikzset{
  bigblue/.style={circle, draw=blue!80,fill=blue!40,thick, inner sep=1.5pt, minimum size=5mm},
  bigred/.style={circle, draw=red!80,fill=red!40,thick, inner sep=1.5pt, minimum size=5mm},
  bigblack/.style={circle, draw=black!100,fill=black!40,thick, inner sep=1.5pt, minimum size=5mm},
  bluevertex/.style={circle, draw=blue!100,fill=blue!100,thick, inner sep=0pt, minimum size=2mm},
  redvertex/.style={circle, draw=red!100,fill=red!100,thick, inner sep=0pt, minimum size=2mm},
  blackvertex/.style={circle, draw=black!100,fill=black!100,thick, inner sep=0pt, minimum size=2mm},  
  whitevertex/.style={circle, draw=black!100,fill=white!100,thick, inner sep=0pt, minimum size=2mm},  
  smallblack/.style={circle, draw=black!100,fill=black!100,thick, inner sep=0pt, minimum size=1mm},
  smallwhite/.style={circle, draw=black!100,fill=white!100,thick, inner sep=0pt, minimum size=1mm},    
}

\begin{document}

\makeatletter{\renewcommand*{\@makefnmark}{}
\footnotetext{Research of all four authors was supported by the Natural Sciences and Engineering Research Council of Canada.}\makeatother}
\makeatletter{\renewcommand*{\@makefnmark}{}
\footnotetext{This research was completed while the third author was a student at Thompson Rivers University.}\makeatother}

\maketitle

\begin{abstract}
Let $p$ and $q$ be positive integers with $p/q \geq 2$.
The ``reconfiguration problem'' for circular colourings asks, given two  $(p,q)$-colourings $f$ and $g$ of a graph $G$, is it possible to transform $f$ into $g$ by changing the colour of one vertex at a time such that every intermediate mapping is a $(p,q)$-colouring? We show that this problem can be solved in polynomial time for $2\leq p/q <4$ and that it is \complete{PSPACE} for $p/q\geq 4$. This generalizes a known dichotomy theorem for reconfiguring classical graph colourings. As an application of the reconfiguration algorithm, we show that graphs with fewer than $(k-1)!/2$ cycles of length divisible by $k$ are $k$-colourable.
\end{abstract}

\section{Introduction}
In recent years, a large body of research has emerged concerning so called ``reconfiguration'' variants of combinatorial problems (see, e.g., the survey of van den Heuvel~\cite{survey} and the references therein, and well as~\cite{Ito2011,Ito2014}). These problems are typically of the following form: given two solutions to a fixed combinatorial problem (e.g. two cliques of order $k$ in a graph or two satisfying assignments of a specific \class{$3$-SAT} instance) is it possible to transform one of these solutions into the other by applying a sequence of allowed modifications such that every intermediate object is also a solution to the problem? 

As a specific example, for a fixed integer $k$ and a graph $G$, one may ask the following: given two (proper) $k$-colourings $f$ and $g$ of $G$, is it possible to transform $f$ into $g$ by changing the colour of one vertex at a time such that every intermediate mapping is a $k$-colouring?\footnote{Throughout the paper, the term \emph{$k$-colouring} will refer to a proper $k$-colouring.} In the affirmative we say that $f$ \emph{reconfigures} to $g$. This problem is clearly solvable in polynomial time for $k\leq 2$. Rather surprisingly, Cereceda, van den Heuvel and Johnson~\cite{3col} proved that it is also solvable in polynomial time for $k=3$ despite the fact that determining if a graph admits a $3$-colouring is \complete{NP}. On the other hand, Bonsma and Cereceda~\cite{Bonsma} proved that, when $k\geq 4$, the problem is \complete{PSPACE}. (As pointed out in~\cite{CerecedaThesis}, a similar result was proved by Jakob~\cite{Jakob}, but in his result the integer $k$ is part of the input.) Combining these two results yields the following dichotomy theorem:

\begin{thm}[Cereceda, van den Heuvel and Johnson~\cite{3col}; Bonsma and Cereceda~\cite{Bonsma}]
\label{colourings}
The reconfiguration problem for $k$-colourings is solvable in polynomial time for $k\leq 3$ and is \complete{PSPACE} for $k\geq 4$. 
\end{thm} 

In this paper, we study the complexity of the reconfiguration problem for circular colourings. Given a graph $G$ and positive integers $p$ and $q$ with $p/q\geq2$, a \emph{(circular) $(p,q)$-colouring} of $G$ is a mapping $f:V(G)\to \{0,\dots,p-1\}$ such that 
\begin{equation}\label{circCond}\text{if $uv\in E(G)$, then $q\leq |f(u)-f(v)|\leq p-q$.}\end{equation} 
Clearly, a $(p,1)$-colouring is nothing more than a $p$-colouring and so $(p,q)$-colourings generalize classical graph colourings. Circular colourings were introduced by Vince~\cite{Vince}, and have been studied extensively; see the survey of Zhu~\cite{Zhusurvey}. 
Analogous to that of classical graph colourings, the reconfiguration problem for circular colourings asks, given $(p,q)$-colourings $f$ and $g$ of $G$, whether it is possible to reconfigure $f$ into $g$ by recolouring one vertex at a time while maintaining (\ref{circCond}) throughout. 

Classical graph colourings and circular colourings are both special cases of graph homomorphisms.  Recall, a \emph{homomorphism} from a graph $G$ to a graph $H$ (also called an \emph{$H$-colouring} of $G$) is a mapping $f:V(G)\to V(H)$ such that $f(u)f(v)\in E(H)$ whenever $uv\in E(G)$. The notation $f:G\to H$ indicates that $f$ is a homomorphism from $G$ to $H$. In this language, a $k$-colouring is simply a homomorphism to a complete graph on $k$ vertices. A $(p,q)$-colouring of $G$ is equivalent to a homomorphism from $G$ to the graph $G_{p,q}$ which has vertex set $\{0,\dots,p-1\}$ and edge set $\{ij: q\leq |i-j|\leq p-q\}$. The graph $G_{p,q}$ is called a \emph{circular clique}. 

\begin{rem}
It is well known that $G_{p,q}$ admits a homomorphism to $G_{p',q'}$ if and only if $p/q \leq p'/q'$~\cite{Vince}. Therefore, since the composition of two homomorphisms is a homomorphism, a graph $G$ admits a $(p,q)$-colouring if and only if it admits a $(p',q')$-colouring for all $p'/q'\geq p/q$. 
\end{rem}

Given two homomorphisms $f,g : G \to H$, we say $f$ \emph{reconfigures} to $g$ if there a sequence $(f=f_0), f_1, f_2, \dots, (f_n=g)$ of homomorphisms from $G$ to $H$ such that $f_i$ and $f_{i+1}$ differ on only one vertex. The sequence is referred to as a \emph{reconfiguration sequence}.   Clearly the existence of a reconfiguration sequence from $f$ to $g$ can be determined independently for each component of $G$, so we may assume that $G$ is connected.  We define the general homomorphism reconfiguration problem as follows.  Let $H$ be a fixed graph.

\subsubsection*{\texorpdfstring{$\boldsymbol{H}$}{H}-\textsc{Recolouring}}

\begin{description}\itemsep -3pt
  \item[Instance:] A connected graph $G$, and two homomorphisms $f,g : G \to H$.
  \item[Question:] Does $f$ reconfigure to $g$?
\end{description}

When $H=K_k$ or $H=G_{p,q}$ we will call the problem $k$-\textsc{Recolouring}
and $(p,q)$-\textsc{Recolouring} respectively.  Thus, Theorem~\ref{colourings} is a dichotomy theorem for $k$-\textsc{Recolouring}.  Our main result is a dichotomy theorem for $(p,q)$-\textsc{Recolouring}:

\begin{thm}
\label{circularThm}
Let $p,q$ be fixed positive integers with $p/q \geq 2$.  Then the $(p,q)$-\textsc{Recolouring} problem is solvable in polynomial time for $2\leq p/q <4$ and is \complete{PSPACE} for $p/q\geq 4$. 
\end{thm}

The complexity of $H$-\textsc{Recolouring} is only known for a handful of families of targets. Theorem~\ref{colourings} is a dichotomy theorem for the family of complete graphs and Theorem~\ref{circularThm} is a dichotomy theorem for the family of circular cliques. Recently, Wrochna~\cite{C4free} (see also~\cite{Marcin}) proved that $H$-\textsc{Recolouring} is polynomial whenever $H$ does not contain a $4$-cycle. In contrast, one can observe that $G_{p,q}$ contains $4$-cycles whenever $p>2q+1$ and so the polynomial side of Theorem~\ref{circularThm} does not follow directly from the result of Wrochna. In a follow-up paper~\cite{followUp}, we determine the complexity of $H$-\textsc{Recolouring} for several additional classes of graphs including, for example, odd wheels.

The rest of the paper is outlined as follows. In Section~\ref{circularPoly}, we provide an explicit polynomial-time algorithm for deciding the $(p,q)$-\textsc{Recolouring} problem when $2\leq p/q <4$. In Section~\ref{circularPSPACE}, we show that, when $p/q\geq 4$, the reconfiguration problem for $\left\lfloor p/q\right\rfloor$-colourings can be reduced to the reconfiguration problem for $(p,q)$-colourings, thereby completing the proof of Theorem~\ref{circularThm} (via Theorem~\ref{colourings}). We close the paper by presenting an unpublished argument of Wrochna which uses a result of~\cite{3col} (on which our algorithm is based) to show that graphs with no cycle of length $0\bmod 3$ are $3$-colourable. This result was originally proved by Chen and Saito~\cite{Chen1994}. We then generalize Wrochna's argument to show that graphs with chromatic number greater than $k$ must contain at least $\frac{(k-1)!}{2}$ distinct cycles of length $0\bmod k$ and conjecture a stronger bound.

\section{The Polynomial Cases: \texorpdfstring{$\boldsymbol{2 \leq p/q < 4}$}{2 <= p/q < 4}} 
\label{circularPoly}

We now extend the ideas of~\cite{3col} to study the complexity of $(p,q)$-\textsc{Recolouring} for $2 \leq p/q < 4$.  Given two $3$-colourings $f$ and $g$ of a graph $G$, the algorithm in~\cite{3col} consists of two phases.   The first phase tests whether the so-called ``fixed vertices'' (vertices that cannot be recoloured under \emph{any} sequence of recolourings) are assigned the same colours under $f$ and $g$.  If not, then we know that $f$ does not reconfigure to $g$ and the algorithm terminates. 

If the algorithm does not terminate at this point, then it enters the second phase. In this phase, the algorithm obtains two edge labellings of $G$ from the vertex colourings $f$ and $g$. The key property of these labellings is that, given that the algorithm did not terminate after the first phase, one can show that $f$ reconfigures to $g$ if and only if every cycle of $G$ has the same ``weight'' under both edge labellings. After a polynomial number of steps, the algorithm either produces a reconfiguration sequence or discovers a cycle with different weights under $f$ and $g$.  We follow a similar approach, but we begin by examining the edge relabelling problem which turns out to depend only on cycle weights.  

\begin{rem}
The work in~\cite{3col} has been extended in~\cite{Johnson2014} where the authors find shortest paths between 3-colourings.  In our work, we do not consider shortest paths, but rather are simply concerned with testing the existence of paths in polynomial time. 
\end{rem}

\subsection{Edge Relabellings}


Let $f$ be a $(p,q)$-colouring of a graph $G$.  Let $\overrightarrow{G}$ be an oriented graph obtained by arbitrarily orienting each edge of $G$. (This orientation is required to define the weighting of paths, cycles, and cuts below.  Thus throughout this section we will assume a graph $G$ has an orientation $\overrightarrow{G}$.)
The \emph{edge-labelling induced by $f$} is defined as $\varphi_f : E(G) \to \{ 0, 1, \dots, p-1 \}$ by $\varphi_f(e) = f(v) - f(u) \bmod{p}$ where $e = \overrightarrow{uv}$.  Let $C$ be a cycle of $G$ and arbitrarily choose a direction of traversal.  Let $\for{C}$ be those edges of $C$ whose orientation in $\overrightarrow{G}$ agrees with the direction of traversal (forward arcs) and $\bac{C}$ be those edges of $C$ whose orientation in $\overrightarrow{G}$ is reversed to the direction of traversal (backward arcs).  Define
\begin{equation}\label{eqn:cyclesum}
\varphi_f(C) = \sum_{e \in \for{C}} \varphi_f(e) + \sum_{e \in \bac{C}} (p - \varphi_f(e)).
\end{equation}
Note that the summation above is in $\mathbb{Z}$, i.e. it is not reduced modulo $p$.  The following properties of $\varphi_f$ are immediate from the fact that $f$ is a $(p,q)$-colouring and the sum $\varphi_f(C)$ is a telescoping series in the values of $f$ on $C$.
\begin{description}
  \item[P1] $q \leq \varphi_f(e) \leq p-q$ for all edges $e$.
  \item[P2] $\varphi_f(C) \equiv 0 \pmod{p}$ for all cycles $C$.
\end{description}
Given a graph $G$, an edge-labelling $\psi : E(G) \to \{ 0, 1, \dots, p-1 \}$ is a \emph{$(p,q)$-labelling} if it satisfies properties \textbf{P1} and \textbf{P2} above.  
Similar to the weighting function for cycles in~\eqref{eqn:cyclesum}, we define a weighting for paths: $\varphi(P)$ is the sum of $\varphi(e)$ and $p-\varphi(e)$ for forward and backward arcs $e$ in $P$, respectively, according to some chose direction of traversal on $P$.

We now introduce the reconfiguration process for edge-labellings.  Let $G$ be a graph and $\overrightarrow{G}$ an orientation of $G$.  For $\emptyset \neq X \subsetneq V(G)$, the \emph{edge cut} $\partial(X)$ is the set of edges with one end in $X$
and the other in $\overline{X}$. Let $\partial^+(X)$ (resp. $\partial^-(X)$) be those edges oriented from $X$ to $\overline{X}$ (resp. $\overline{X}$ to $X$) in $\overrightarrow{G}$. Given a $(p,q)$-labelling $\varphi: E(G) \to \{ 0, 1, \dots, p-1 \}$, let $\alpha$
be an integer $1 \leq \alpha \leq p-1$, where $\varphi(e) \geq q + \alpha$ for $e \in \partial^+(X)$ and $\varphi(e) \leq p-q-\alpha$ for $e \in \partial^-(X)$. The labelling $\varphi'$ obtained from $\varphi$ by \emph{relabelling on $\partial(X)$ by $\alpha$} is defined by:
$$
\varphi'(e) = \left\{ \begin{array}{ll} 
\varphi(e) - \alpha & \mbox{ if } e \in \partial^+(X), \\ 
\varphi(e) + \alpha & \mbox{ if } e \in \partial^-(X) \end{array} \right.
$$
This process is called an \emph{edge cut relabelling}.  Clearly if $\varphi$ is a $(p,q)$-labelling, then $\varphi'$ is as well.  It is easy to verify that for any cycle $C$, $\varphi(C) = \varphi'(C)$.

We remark an alternative relabelling definition is to simply require that if $\varphi$ is a $(p,q)$-labelling, then $\varphi'$ is as well (with no requirement that $\varphi(e) \geq q + \alpha$ for $e \in \partial^+(X)$ and $\varphi(e) \leq p-q-\alpha$ for $e \in \partial^-(X)$).  In this case one would naturally reduce $\varphi'$ modulo $p$ after shifting by $\alpha$.  With this more general definition the weight of a cycle can change after relabelling.  For example, consider a directed $4$-cycle with $p=4$ and $q=1$.  The constant labelling of $1$ on each edge is a $(4,1)$-labelling.  The weight of this cycle is $4$ (when traversed in the direction of the edges).  Using $\alpha=2$, one can relabel to obtain the constant labelling $\varphi'(e)=3$ for all edges: simply add $2$ to the first and third edges, and $-2$ to the second and fourth edges.  After reducing the weights modulo $4$, we see $\varphi'(e)=3$ for all edges.  The weight of the cycle is now $12$. Nonetheless, the two notions of relabelling are equivalent with the assumption $p/q < 4$ and the natural restriction $-p/2 \leq \alpha \leq p/2$.  We will use the first definition as it eases the analysis below.


\begin{thm}\label{thm:cocycle}
Let $G$ be a connected graph and $\varphi, \psi$ be two $(p,q)$-edge-labellings of $G$.  Then $\varphi$ reconfigures to $\psi$ through a sequence of edge cut relabellings if and only if $\varphi(C) = \psi(C)$ for all cycles $C$ in $G$. 
\end{thm}

\begin{proof}
Relabelling on an edge cut does not change the weight of any cycles.  Thus, if $\varphi$ reconfigures to $\psi$, then $\varphi(C) = \psi(C)$ for all cycles $C$ is a necessary condition.  To prove sufficiency, assume $\varphi(C) = \psi(C)$ for each cycle $C$. We may assume $\varphi \neq \psi$. Let $\overrightarrow{wu}$ be an arc such that $\varphi(wu) < \psi(wu)$. (The case where $\varphi(wu) > \psi(wu)$ is analogous.) Let $T$ be a spanning tree of $G$ rooted at $u$ with all arcs directed away from $u$.  For each $v \in V(G)$ define $\wt(v,\varphi) = \varphi(P)$ where $P$ is the $(u, v)$ path in $T$, traversed from $u$ to $v$.  Similarly define $\wt(v,\psi)$. 

Let $X = \{ v : \wt(v, \varphi) \leq \wt(v, \psi) \}$.  Thus, $\overline{X} = \{ v :  \wt(v, \varphi) > \wt(v, \psi) \}$.  Clearly $u \in X$ and it is easy to see $w \in \overline{X}$.  Suppose to the contrary $w \in X$.  Then $uw$ is not an edge of $T$ and thus the $uw$ path in $T$ plus the edge $wu$ is a cycle $C$.  Moreover this cycle must have $\varphi(C) < \psi(C)$ when  traversed with the direction of the path $P$, contrary to our assumption.  Hence $\emptyset \neq X \subsetneq V(G)$ and $\partial(X)$ is a well defined edge cut.

We claim that $\partial^+(X)$ must only contain arcs where $\varphi(e) > \psi(e)$ and arcs in $\partial^-(X)$ must have $\varphi(e) < \psi(e)$.  Suppose to the contrary there is an edge $e=\overrightarrow{xy}, x \in X, y \in \overline{X}$ such that $\varphi(e) \leq \psi(e)$.  (The proof for arcs in $\partial^-(X)$ analogous.)  By construction there is a $(u,x)$-path $P_1$ such that $\varphi(P_1) \leq \psi(P_1)$.  Thus $P_1+y$ is an $(u,y)$-path such that $\varphi(P_1+y) \leq \psi(P_1+y)$.  On the other hand, since $y \in \overline{X}$, the path $(u,y)$ path in $T$, say $P_2$, has the property that $\varphi(P_2) > \psi(P_2)$.  Hence, reverse of the path is a $(y,u)$-path, $P_2^R$ such that $\varphi(P_2^R) < \psi(P_2^R)$.  The concatenation of $P_1+y$ and $P_2^R$ is a closed $(u,u)$-walk with less weight under $\varphi$ than $\psi$.  In particular, it must contain a cycle $C$ such that $\varphi(C) \neq \psi(C)$, a contradiction. Let $\alpha = \min_{e \in \partial(X)} \{ |\varphi(e)-\psi(e)| \}$.  Relabel $\partial(X)$ by $\alpha$ to obtain $\varphi'$.  Now $\varphi'$ agrees with $\psi$ on more edges than $\varphi$. The result follows.
\end{proof}

Observe the second half of the proof of Theorem~\ref{thm:cocycle} shows that if there is a cycle $C$ with different weights under $\varphi$ and $\psi$, then $C$ is the fundamental cycle in $T+e$ where $e \in \partial(X)$.  The discovery of such a cycle through examining fundamental cycles plays a key role in our polynomial time algorithm.

\begin{cor}\label{cor:fundamentalCycle}
Let $G$ be a graph together with two $(p,q)$-edge-labellings $\varphi$ and $\psi$. Either $\varphi$ reconfigures to $\psi$ through a sequence of edge cut relabellings or there is a cycle $C$ for which $\varphi(C) \neq \psi(C)$.  Morever, in the latter case $C$ may be taken to be the fundamental cycle in $T+e$ where $T$ is a fixed spanning tree of $G$, and $e \in \partial(X)$ as defined in the proof of Theorem~\ref{thm:cocycle}. Determining which of the two cases (exclusively) holds can be determined in polynomial time.  
\end{cor}

\subsection{Vertex Recolouring}

We now return to recolouring vertices.  We have already observed that given a $(p,q)$-colouring $f$ of a graph, we can construct a $(p,q)$-edge labelling $\varphi_f$.  Conversely given $\varphi$, a $(p,q)$-edge labelling, using the ideas from the proof of Theorem~\ref{thm:cocycle}, we can generate a weight, $\wt(v,\varphi)$, for each vertex $v$.  Reducing these weights modulo $p$ yields a $(p,q)$-colouring $f$ where $\varphi = \varphi_f$.  The following proposition is immediate. Throughout this section addition (for shifting colours) is done modulo $p$.

\begin{prop}\label{prop:inducedlabelling}
Given a connected graph $G$, an edge-labelling $\psi : E(G) \to \{ 0, 1, \dots, p-1 \}$ is a $(p,q)$-labelling if and only if $\psi = \varphi_f$ for some $f : G \to G_{p,q}$.  Moreover, if $\varphi_f = \varphi_g$ for $f,g: G \to G_{p,q}$, then there exists $k\in\{0,1,\dots,p-1\}$ such that $f(v)=g(v) + k$ for all $v\in V(G)$.
\end{prop}

To connect our work on edge cut relabellings back to vertex recolourings, we must realize edge cut relabelling through vertex recolourings where we recolour one vertex at a time.  First, we observe that relabelling $\varphi$ on $\partial(X)$ by $\alpha$ corresponds to the following vertex recolouring where we recolour an entire set of vertices (simultaneously).  Let $f'$ be the colouring obtained from $f$ by:
$$
f'(v) = \left\{ \begin{array}{ll} f(v)+\alpha & \mbox{ if } v \in X, \\
                                  f(v) & \mbox{ otherwise}.
                \end{array} \right.
$$
Then $\varphi_{f'}$ is the edge-labelling obtained from $\varphi_f$ by relabelling on $\partial(X)$ by $\alpha$.  We call this process \emph{recolouring the vertex set $X$ by $\alpha$}. Provided $\varphi(e) \geq q+\alpha$ for all $e \in \partial^+(X)$ and $\varphi(e) \leq p-q-\alpha$ for all $e \in \partial^-(X)$, the new colouring $f'$ is a proper $(p,q)$-colouring.  The following corollary is immediate from Proposition~\ref{prop:inducedlabelling} and Theorem~\ref{thm:cocycle}.

\begin{cor}\label{cor:cyclewts}
Let $f$ and $g$ be two $(p,q)$-colourings of a graph $G$. Then there is a sequence of colourings $f, f_1, f_2, \dots, f_n$, each obtained from its predecessor by a vertex set recolouring, where $f_n(v) = g(v) + k$ for all vertices $v$ and some constant $k$ if and only if $\varphi_f(C) = \varphi_g(C)$ for all cycles $C$ in $G$. 
\end{cor}

To complete our algorithm, we need to determine when recolouring a vertex set $X$ by $\alpha$ can be realized by a sequence of single vertex recolourings.  We begin by proving that for $p/q < 4$, recolouring by $\alpha$ can always be achieved through recolouring $X$ by $1$ ($\alpha$ times). Note this not the case when $p/q \geq 4$.  For example, consider a $(4,1)$-colouring of $C_4$ where the vertices are coloured $0,1,2,3$ (in their cyclic order).  The vertex with colour $0$ can be recoloured to $2$.  However, the recolouring cannot be done by increasing the colour by $1$ (twice).

Given a pair $a,b\in\{0,\dots,p-1\}$, the \emph{interval} $[a,b]$ is defined to be the set $\{a,a+1,\dots, b-1,b\}$ (where addition is modulo $p$). An important property of $G_{p,q}$ used in the proof below is that, when $2\leq p/q<4$, the common neighbours of any two vertices form an interval.  This fact, and its proof, appear in~\cite{BrewsterNoel}. (Roughly, if $x$ and $y$ are common neighbours in two non-contiguous intervals, then the non-neighbours of $x$ and the non-neighbours of $y$ are disjoint intervals.  This requires $p \geq 2+2 \cdot (2q-1)$ or $p/q \geq 4$.)

\begin{prop}\label{prop:oneAtATime}
For $2 \leq p/q < 4$, let $f, g$ be two $(p,q)$-colourings of a graph $G$ where $g$ is obtained from $f$ by recolouring the vertex set $X$ by $\alpha$. Then there is a sequence of colourings $(f = f_0), f_1, f_2, \dots, (f_\alpha = g)$ where each colouring is obtained from its predecessor by recolouring $X$ by $1$ (for the entire sequence) or each is obtained from its predecessor by recolouring $X$ by $-1$ (for the entire sequence).
\end{prop}

\begin{proof}
Let $v \in X$ and $vu$ be an edge.  If $u \in X$, then $(f(v)+1)(f(u)+1)$ and $(f(v)-1)(f(u)-1)$ are both edges of $G_{p,q}$.  Consequently we can increment the colour by $1$ of each vertex in $X$, or decrement the colour by $1$ of each vertex in $X$, and maintain a proper $(p,q)$-colouring on $X$.

On the other hand, consider a neighbour $u$ of $v$ such that $u \not\in X$.  Note $f(u)f(v)$ and $f(u)(f(v)+\alpha)$ are both edges of $G_{p,q}$.  If $\alpha=p/2$, then $f(u)$ is distance at least $q$ from both $f(v)$ and $f(v)+p/2$ which implies $p/2 \geq 2q$, contrary to our assumption.  If $\alpha < p/2 < 2q$, then $f(u)$ must belong to the interval $[f(v)+\alpha, f(v)]$. 
Thus $f(u)$ is adjacent in $G_{p,q}$ to every vertex of $[f(v), f(v)+\alpha]$ and in particular to $f(v)+1$.  Note this argument depends only on $\alpha < p/2$.  Hence the colouring $f_1$ obtained from $f$ by increasing the colour of each vertex in $X$ by $1$ is a $(p,q)$-colouring of $G$.

A similar argument shows for $\alpha > p/2$, the colouring $f_1$ obtained from $f$ by decreasing the colour of each vertex in $X$ by $1$ is a $(p,q)$-colouring.  The result follows.
\end{proof}

Given $G$, an orientation $\overrightarrow{G}$, and an $(p,q)$-edge-labelling, say $\varphi$, let $D_\varphi$ be the subdigraph of $\overrightarrow{G}$ induced by edges with label $q$ or $p-q$.  Without loss of generality we may assumed $D_\varphi$ is oriented so that all edges have label $q$.  (To simplify notation we shall write $D_f$ instead of $D_{\varphi_f}$ when the edge-labelling $\varphi_f$ comes from some vertex colouring $f$.)  In the case $p/q = 2$, each edge will be oriented both ways, thus we have a directed 2-cycle on each edge.

\begin{lem}\label{lem:starcuts}
Let $f$ be a $(p,q)$-colouring of a graph $G$. Let $f'$ be the $(p,q)$-colouring obtained from $f$ by recolouring some vertex set $X$ by $1$.  Then $f'$ can obtain from $f$ by a sequence of recolourings of single vertices if and only if induced subdigraph $D_f[X]$ contains no directed cycles. 
\end{lem}

\begin{proof}
Suppose $D_f[X]$ contains no directed cycles.  We topologically sort the vertices of $X$ to obtain an ordering $x_1, x_2, \dots, x_t$ where $e = \overrightarrow{x_i x_j} \in E(D_f[X])$ implies $i < j$.  For each $i= t, t-1, \dots, 1$ increase $f(x_i)$ by $1$.  

On the other hand, suppose $D_f[X]$ contains a directed cycle.  If one could sequentially increase the colour of each vertex in $X$ by $1$, then there would be a first vertex $c_i$ in the cycle to have its colour changed.  However, $c_i$ has an out neighbour $c_{i+1}$ (in $D_f[X]$) meaning $f(c_{i+1})-f(c_i) = q$.  The resulting colouring gives the edge $c_i c_{i+1}$ a weight of $q-1$, i.e. the resulting colouring is a not a proper $(p,q)$-colouring.
\end{proof}

\subsection{Polynomial Time Algorithm}

We now describe the polynomial time algorithm for recolouring.  We present a proof of correctness for each step below. The algorithm itself is in Figure~\ref{fig:Algorithm}.

Let $f$ and $g$ be two $(p,q)$-colourings of a graph $G$. As defined in~\cite{3col}, vertices whose colours can never change under any sequence of recolourings are called \emph{fixed}. Following the ideas of~\cite{3col} and using the results above, we identify (the only) three obstructions preventing the recolouring of $f$ to $g$: fixed vertices, cycle weights, and weights of paths between fixed vertices.  

The first step of the algorithm is to identify fixed vertices and verify $f(v)=g(v)$ for all fixed vertices. 

\begin{lem}\label{lem:fixedVertices}
Let $f$ be a $(p,q)$-colouring of a graph $G$ where $2 \leq p/q < 4$.  Suppose $v$ is a vertex in a strongly connected component of $D_f$.  Then the colour of $v$ cannot be changed under any sequence of recolourings.
\end{lem}

\begin{proof}
Since $v$ belongs to a strongly connected component of $D_f$, $v$ must belong to a directed cycle in $D_f$.  By definition of $D_f$, the predecessor and successor of $v$ on this cycle receive colours $f(v)-q$ and $f(v)+q$ respectively.  As noted above, and proved in~\cite{BrewsterNoel}, the common neighbours of two vertices in $G_{p,q}$ form an interval.  As the colour of $v$ must be a common neighbour of $f(v)-q$ and $f(v)+q$, the only choice for $v$ is $f(v)$. That is, the colour of $v$ cannot change until the colour of one its neighbours changes.  However, this is true for every vertex on the directed cycle.  All the vertices in the strongly connected component are fixed.
\end{proof}


The strongly connected components of $D_f$ can be found in linear time; hence, we can find the fixed vertices and verify equality of $f$ and $g$ on the fixed vertices in linear time.  (See, e.g., \cite{AhoHU83}.)

In Step 2 (a) of the algorithm we construct a spanning tree $T$ rooted a vertex $u$, and construct the cut $\partial(X)$ as in the proof of Theorem~\ref{thm:cocycle}. Corollary~\ref{cor:fundamentalCycle} shows if some cycle has different weight under $\varphi_f$ and $\varphi_g$, then we will discover a fundamental cycle in $T+e$ for some edge $e \in \partial(X)$.

In Step 2 (b) we calculate the shift $\alpha$ that will increase the number of edges upon which $\varphi_f$ and $\varphi_g$ agree. Proposition~\ref{prop:oneAtATime}, Lemma~\ref{lem:starcuts}, and Lemma~\ref{lem:fixedPath} show that the vertices of $X$ or $\overline{X}$ can be recoloured $1$ at a time, or there is a path between fixed vertices certifying that $f$ does not recolour to $g$. Specifically, in the case we wish to increase the vertex set $X$ by $1$ and $D_f[X]$ contains a directed cycle, we can achieve the same change to the edge labelling on $\partial(X)$ by decreasing $\overline{X}$ by $1$ provided $D_f[\overline{X}]$ does not contain a directed cycle. In the event both subgraphs contain a directed cycle, we have an obstruction to recolouring as described in the following lemma.

%
%
%
%
%
%

\begin{lem}\label{lem:fixedPath}
Let $G$ be a graph with two $(p,q)$-colourings $f$ and $g$.  Suppose $T$ is a spanning tree rooted at $u$. Let $X = \{ v : \wt(v, \varphi_f) \leq \wt(v, \varphi_g) \}$ (as defined in Theorem~\ref{thm:cocycle}).  If $D_f[X]$ and $D_f[\overline{X}]$ both contain a directed cycle, then $f$ does not recolour to $g$.
\end{lem}

\begin{proof}
Let $x$ belong to a directed cycle in $X$ and $y$ belong to a directed cycle in $\overline{X}$.  Then there are paths in $T$ from $u$ to $x$, say $P_1$, and from $u$ to $y$, say $P_2$ such that $\varphi_f(P_1) \leq \varphi_g(P_1)$ and $\varphi_f(P_2) > \varphi_g(P_2)$.  Thus the reverse of $P_1$ concatenated with $P_2$ is an $x,y$-path such that $\varphi_f(P_1^RP_2) > \varphi_g(P_1^RP_2)$.  Since the end points of this path belong to directed cycles in $D_f$, their colours are fixed.  Let $z$ be an internal vertex of the path.  If $z$ can be recoloured by $\alpha$ it is easy to check the sum of the weight of the two path edges incident with $z$ does not change.  (For example, if both edges are directed in the direction of traversal of the path, one increases by $\alpha$ and the other decreases by $\alpha$.  The analyses for the other possible orientations of the two edges are similar.)
\end{proof}


In~\cite{3col}, fixed vertices are found by successively deleting sources and sinks from $D_f$.  Thus vertices belonging to a directed cycle will be fixed, but also vertices on a directed path between two directed cycles are also fixed.  In our work, we identify fixed vertices of the former type using strongly connected components of $D_f$ and vertices of the latter type by the weight of paths between strongly connected components (our third obstruction) as directed paths in $D_f$ have the smallest possible weight over all edge labellings.  

In Step 3 of the algorithm, we have recoloured vertices to obtain a colouring $f_n$ such that $\varphi_{f_n} = \varphi_g$.  By Corollary~\ref{cor:cyclewts}, $f_n(v) = g(v)+k$ for some constant $k$.  If $D_f$ contains any directed cycles, then there are fixed vertices under $f$.  In Step 1 we have checked that $f(v) = g(v)$ for fixed vertices, so we can conclude $k=0$ and $f_n = g$.  On the other hand, if $k \neq 0$, then there are no directed cycles in $D_f$.  Applying Proposition~\ref{prop:oneAtATime} with $X = V(G)$ shows we can recolour the vertices one at a time to reconfigure $f_n$ to $g$.

\begin{figure}\label{fig:Algorithm}

\centerline{\textbf{The Recolouring Algorithm}}

Let $2 \leq p/q < 4$.

\begin{description}
  \item[Input:] A graph $G$ and two $(p,q)$-colourings $f, g$.
  \item[Output:] A recolouring sequence from $f$ to $g$ or an obstruction to recolouring.
\end{description}

\begin{enumerate}
  \item Find the strongly connected components of $D_f$.  For each vertex $v$ belonging to a nontrivial strongly connected component, verify $f(v) = g(v)$.  If for some such $v$, $f(v) \neq g(v)$, then answer NO, return a directed cycle to which $v$ belongs, and STOP.
  \item Construct a spanning tree $T$ rooted at a vertex $u$. Partition $V(G)$ into three sets: $X_\ell, X_e, X_s$ consisting of those vertices $v$ whose $(u,v)$-path in $T$ has weight larger, equal, smaller under $\varphi_f$ versus $\varphi_g$ respectively.  Repeat until $X_\ell \cup X_s = \emptyset$:
  \begin{enumerate}
    \item If $X_\ell \neq \emptyset$, then let $X = X_e \cup X_s$ and $\overline{X} = X_\ell$. (If $X_\ell = \emptyset$ and
    $X_s \neq \emptyset$ apply an analogous process reversing their roles.) For each $e \in \partial^+(X)$ (resp. $\partial^-(X)$), verify $\varphi_f(e) > \varphi_g(e)$ (resp. $\varphi_f(e) < \varphi_g(e)$).  If some edge fails this test, then answer NO, return a cycle with different weights under $\varphi_f$ and $\varphi_g$, and STOP.
    \item If $D_f[X]$ and $D_f[\overline{X}]$ both contain nontrivial strongly connected components, then answer NO, return a path $P$ between two fixed vertices with $\varphi_f(P) \neq \varphi_g(P)$, and STOP.
    \item Let $\alpha = \min_{e \in \partial(X)} |\varphi_f(e) - \varphi_g(e)|$.  Recolour the vertices of $X$ (or $\overline{X}$) by $1$ using a sequence of single vertex recolourings. Repeat this recolouring by $X$ (by $1$) $\alpha$ times.
    \item Update the path weights in $T$ and the sets $X_\ell, X_e, X_s$.
  \end{enumerate}
  \item At this point $\varphi_f = \varphi_g$.  If $f(u) = g(u)$, then answer YES, return the sequence of recolourings, and STOP.  Otherwise, we know $G$ contains no directed cycles and we can topologically sort $V(G): v_1, v_2, \dots v_n$.  Increase (or decrease) the colour of all the vertices by $1$ in the order $v_n$ to $v_1$.  Repeat this until $f(u) = g(u)$.  Answer YES, return the recolouring sequence, and STOP.
\end{enumerate}

\caption{The Recolouring Algorithm}
\end{figure}

The following theorem is immediate from the algorithm.

\begin{thm}\label{thm:char}
Let $2 \leq p/q < 4$.  Suppose $f$ and $g$ are $(p,q)$-colourings of a graph $G$, $\overrightarrow{G}$ is an orientation of $G$ and $\varphi_f, \varphi_g$ are the edge labellings of $G$ obtained from $f$ and $g$.  Then $f$ recolours to $g$ if and only if: 
\begin{enumerate}
  \item the same set of vertices are fixed under $f$ and $g$ and $f(v) = g(v)$ for all fixed vertices;
  \item for each cycle $C$ of $G$, $\varphi_f(C) = \varphi_g(C)$; and
  \item for each path $P$ whose end points are fixed, $\varphi_f(P)=\varphi_g(P)$.
\end{enumerate}
Moreover, one can find a recolouring sequence from $f$ to $g$ or an obstruction of the above type in polynomial time.
\end{thm}

At each iteration of the algorithm we recolour $O(|V(G)|)$ vertices and increase $X_e$ by at least one vertex.  Once a vertex is in $X_e$ it never leaves.  Thus we do $O(|V(G)|)$ recolouring steps.
\begin{cor}
Let $2 \leq p/q < 4$.  Suppose $f$ and $g$ are $(p,q)$-colourings of a graph $G$.  If $f$ recolours to $g$, then there is a reconfiguration sequence of length $O(|V(G)|^2)$ which certifies this.
\end{cor}

\section{The \complete{{\bf PSPACE}} Cases: \texorpdfstring{$\boldsymbol{p/q\geq 4}$}{p/q >= 4}}
\label{circularPSPACE}

\subsection{Overview}

As in the previous section, all addition and subtraction involving elements of $\{0,\dots,p-1\}$ is viewed modulo $p$. We say that $i,j\in\{0,\dots,p-1\}$ are \emph{compatible} if they are adjacent in $G_{p,q}$.  For notational convenience, let us define $k:=\lfloor p/q\rfloor$ and $r:= p-kq$. 

Consider, for a moment, the case $r=0$. If $q=1$, then the complexity is \complete{PSPACE} by Theorem~\ref{colourings}. Otherwise, let $\gamma:\{0,\dots,k-1\}\to \{0,\dots,p-1\}$ be defined by $\gamma(i):=qi$ and $\phi:\{0,\dots,p-1\}\to \{0,\dots,k-1\}$ be defined by $\phi(j) := \left\lfloor j/q\right\rfloor$. It is not hard to see that $\gamma$ is a homomorphism from $K_k$ to $G_{p,q}$ and that $\phi$ is a homomorphism from $G_{p,q}$ to $K_k$. Also, given an instance $(G,f,g)$ of the $k$-\textsc{Recolouring} problem, we have that $f$ reconfigures to $g$ if and only if $\gamma f$ reconfigures to $\gamma g$ as $(p,q)$-colourings. Therefore, the $(p,q)$-\textsc{Recolouring} problem is \complete{PSPACE} if $p/q\geq 4$ and $r=0$. Thus while the proofs below work when $r=0$, to avoid trivialities we assume that $r \geq 1$. 

Given an instance $(G,f,g)$ of the $k$-\textsc{Recolouring} problem, we will construct an instance $(G',\alpha,\beta)$ of the $(p,q)$-\textsc{Recolouring} problem in polynomial time such that:
\begin{itemize}
\item $|V(G')| = O(|V(G)|+|E(G)|)$ and
\item $f$ reconfigures to $g$ if and only if $\alpha$ reconfigures to $\beta$. 
\end{itemize}
It will follow from Theorem~\ref{colourings} that the reconfiguration problem for $(p,q)$-colourings is \complete{PSPACE}, thereby completing the proof of Theorem~\ref{circularThm}. 

We give a brief outline of the ideas behind the construction before moving into the finer details. The first step is to divide the set $\{0,\dots,p-1\}$ into $k$ intervals which we will, in some sense, treat as $k$ separate colours. Define
\[S_0:= [0,q+r-1],\text{ and}\]
\[S_i := [iq+r,(i+1)q+r-1]\text{ for }1\leq i\leq k-1.\]
It is clear that $S_i$ and $S_j$ are disjoint for $i\neq j$ and that $\bigcup_{i=0}^{k-1}S_i = \{0,\dots,p-1\}$. Let $\gamma:\{0,\dots,k-1\}\to \{0,\dots,p-1\}$ be the function which maps $i$ to the left endpoint of $S_i$ for $0\leq i\leq k-1$. For $i\neq j$, it is easily observed that the left endpoint of $S_i$ is adjacent to the left endpoint of $S_j$ in $G_{p,q}$. Therefore, given any $k$-colouring $f:V(G)\to \{0,\dots,k-1\}$, the composition $\gamma f$ is a $(p,q)$-colouring. 

This observation guides part of our construction. The graph $G'$ will contain $G$ as an induced subgraph and the $(p,q)$-colourings $\alpha$ and $\beta$ will be defined so that their restrictions to $G$ will be equal to $\gamma f$ and $\gamma g$, respectively. Now, if it is possible to reconfigure $\alpha$ to $\beta$ in such a way that none of the intermediate colourings map a pair of adjacent vertices of $G$ to the same set $S_i$, then we immediately obtain a reconfiguration sequence taking $f$ to $g$. That is, to obtain a sequence of $k$-colourings taking $f$ to $g$, one could simply compose each $(p,q)$-colouring of the sequence with the function $\varphi:\{0,\dots,p-1\}\to \{0,\dots,k-1\}$ which maps every vertex of $S_i$ to $i$ for $0\leq i\leq k-1$. Notice that, for $i\neq 0$, the set $S_i$ is an independent set in $G_{p,q}$ and therefore no adjacent pair is ever mapped to $S_i$. Thus, all that we need to worry about is that some of the intermediate colourings may map two adjacent vertices of $G$ to $S_0$. To remedy this, we will add some structures (or ``gadgets'') to $G'$ which will forbid adjacent vertices of $G$ from mapping to $S_0$. 

To this end, we start by adding a copy of $G_{p,q}$ disjoint from $G$ with vertex set $\left\{y_0,\dots,y_{p-1}\right\}$, where $y_iy_j$ is an edge whenever $q\leq |i-j|\leq p-q$. We extend the $(p,q)$-colourings $\alpha$ and $\beta$ to $\{y_0,\dots,y_{p-1}\}$ by setting $\alpha(y_i)=\beta(y_i)=i$ for all $i$. It is clear that each vertex $y_i$ is fixed in $\alpha$ and $\beta$ and, moreover, in any $(p,q)$-colouring which reconfigures to $\alpha$ or $\beta$. 

Now, for each edge $uv$ of $G$, we will add two \emph{forbidding paths} $P_{uv}$ and $P_{vu}$ from $u$ to $v$ to $G'$ which are internally disjoint from $V(G)\cup\{y_0,\dots,y_{p-1}\}$ and one another, as well as from every other such path. These paths restrict the colour pairs which can appear on $u$ and $v$ during a reconfiguration process. We achieve this by assigning to each interval vertex of the path, a list of allowed colours for that vertex, i.e. we use a list colouring.  This is accomplished by joining the internal vertices of the paths to specific subsets of $\{y_0,\dots,y_{p-1}\}$.  For example consider the case $p=18, q=4$ which we explore below. Suppose we want a path vertex $x$ to always be coloured with an element of the list $\{ 4, 5, \dots, 11 \}$. Then it suffices to join $x$ to the vertices $\{y_{15}, y_{16}, y_{17}, y_{0}\}$. As we will show, the colours in the lists forbids $u$ and $v$ from mapping to $S_0$. Also, once we describe our construction in detail, it will be clear that the length of the forbidding paths depends only on $p$ and $q$, and so $|V(G')| = O\left(|V(G)| + |E(G)|\right)$.

The difficulty now comes in proving that if $f$ reconfigures to $g$, then $\alpha$ reconfigures to $\beta$. Specifically, given a reconfiguration sequence $(h_i)_{i=1}^s$ taking $f$ to $g$, we need that the lists assigned to the internal vertices of the forbidding paths are flexible enough that we can use $(h_i)_{i=1}^s$ to obtain a reconfiguration sequence taking $\alpha$ to $\beta$. This will be proved using a moderate amount of case analysis at the end of the section. Before moving on, we remark that the general strategy of using some sort of forbidding paths in which the internal vertices are confined to lists was also used by Bonsma and Cereceda~\cite{Bonsma} (in a somewhat different manner) to prove that the reconfiguration problem for $k$-colourings is \complete{PSPACE} for $k\geq4$. 

\subsection{Defining the Forbidding Paths} 

Let $uv$ be an edge of $G$. We will now describe our construction of the forbidding path $P_{uv} = ux_0^{uv}x_1^{uv}\dots x_t^{uv}v$ from $u$ to $v$, including the definition of the lists assigned to the internal vertices $x_0^{uv},\dots,x_t^{uv}$. As mentioned before, the construction of $G'$ also includes a path $P_{vu}$ from $v$ to $u$ which is defined similarly. 

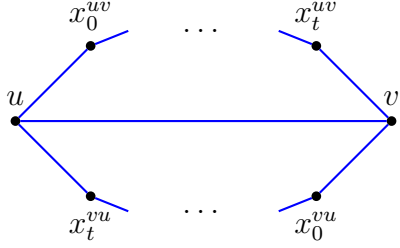
\begin{figure}[htb]
\vspace{-0.5em}
\begin{center}
\begin{tikzpicture}

\node[smallblack,label={$u$}] (u) at (0,1) {};
\node[smallblack,label={$x_0^{uv}$}] (A) at (1,2) {};
\node[smallblack,label=below:{$x_t^{vu}$}] (B) at (1,0) {};

\node[smallblack,label={$v$}] (v) at (5,1) {};
\node[smallblack,label={$x_t^{uv}$}] (Av) at (4,2) {};
\node[smallblack,label=below:{$x_0^{vu}$}] (Bv) at (4,0) {};

\node at (2.5,2.2) {$\dots$};
\node at (2.5,-0.2) {$\dots$};

\draw[thick,blue] (u)--(A)--(1.5,2.2);
\draw[thick,blue] (u)--(B)--(1.5,-0.2);
\draw[thick,blue] (v)--(Av)--(3.5,2.2);
\draw[thick,blue] (v)--(Bv)--(3.5,-0.2);
\draw[thick,blue] (u)--(v);

\end{tikzpicture}
\end{center}
\vspace{-0.5em}
\caption{Forbidding paths attached to each edge $uv$}
\label{fig:paths}
\end{figure}

First, the construction of the lists depends on the sequence $q, 2q, 3q, \dots, r$. Define $t$ to be the smallest positive integer such that $(t+1)q\equiv r \!\!\pmod{p}$. Now, the \emph{lists} for the internal vertices of $P_{uv}$ are intervals of $G_{p,q}$ and are defined as follows:
\[L\left(x_0^{uv}\right) = L\left(x_t^{uv}\right) :=[p-1,2q-1],\]
\[L\left(x_i^{uv}\right):= [iq,(i+2)q-1]\text{ for }1\leq i\leq t-1.\]
As stated in the outline, in order to enforce these lists we add a copy of $G_{p,q}$ on vertex set $\left\{y_0,\dots,y_{p-1}\right\}$ and join the vertices of $P_{uv}\setminus \{u,v\}$ to the appropriate vertices of this set. Specifically, we join the vertices $x_0^{uv}$ and $x_t^{uv}$ to $\left\{y_j: j\in [3q-1,p-q-1]\right\}$ and join $x_i^{uv}$ to $\left\{y_j: j\in [(i+3)q-1,(i-1)q]\right\}$ for $1\leq i\leq t-1$. 

Figure~\ref{fig:pathlist} contains a useful tabular representation of the lists, and a specific case $(p,q)=(18,4)$ is depicted in Figure~\ref{fig:examplelists}. In both of these figures, the colours are laid out so that below colour $c$ in the table is the colour $c+q$. This gives the following proposition.
\begin{prop}\label{prop:downandright}
When colouring the path $x_0^{uv}, x_1^{uv}, \dots, x_t^{uv}$ the colours used (one per row) must be directly below or to the right of the preceding vertex, and any such sequence of colours from the table gives a good colouring of the path with the exception of $L(x_{t-1}^{uv})$ and $L(x_t^{uv})$.  In this case colours in the interval $[r-2q,r-1]$ of $L(x_{t-1}^{uv})$ are compatible with colours below and to the right in the interval $[p-1,q+r]$ of $L(x_t^{uv})$.  Similarly, the interval $[-q-1,r-1]$ is compatible with colours below and to the right of $[p-1,2q-1]$.
\end{prop}

\begin{proof}
For $1 \leq i \leq t-2$ vertex $x_i^{uv}$ receives a colour from the interval $[iq, (i+2)q-1]$ and $x_{i+1}^{uv}$ receives a colour from $[(i+1)q, (i+3)q-1]$.  Observe $iq$ is compatible with all colours in the interval $[(i+1)q,(i-1)q]$.  Since $p+(i-1)q \geq (i+3)q-1$,
we have $iq$ is compatible with every colour on the list for $x_{i+1}^{uv}$.  The colour $iq+1$ is compatible with $[(i+1)q+1,(i-1)q+1]$ which include all colours below and to the right of $iq+1$.  Continuing one sees that $(i+2)q-1$ is compatible with $[(i+3)q-1,(i+1)q-1]$ which includes only the last vertex of the list for $x_{i+1}^{uv}$.  Thus the proposition holds for $1 \leq i \leq t-2$.  It is straightforward to verify the case $i=0$ as well.  Consider now $i=t-1$.  The argument is similar; however, we need to consider the two cases in the statement of the proposition.  First for a colour $[r-2q,r-1]$, a colour in $[p-1,q+r]$ is compatible if and only if it is below and to the right.  (Note the right hand end point $q+r$ is compatible with $r-2q$ if and only if $p+r-2q-q \geq q+r$ or $p \geq 4q$.)  A similar statement holds for $[-q-1,r-1]$ and $[p-1,2q-1]$.  The result follows.
\end{proof}

Finally the vertical bar in the table provides the following information: if vertex $u$ receives colour $0$, then only colours to the right of the vertical bar can be used to colour the path $P_{uv}$.


\begin{figure}[htb]
{\scriptsize
$$
\begin{array}{c|*{6}{c}|*{5}{c}}
\mbox{Vertex} & \multicolumn{11}{c}{\mbox{Lists}} \\ \hline
x_0^{uv}      & p-1 &  0   & \dots  &  q-r-1 & \dots &  q-1  &  q  & \dots  & 2q-1   \\
x_1^{uv}      &     &  q   & \dots  & 2q-r-1 & \dots & 2q-1  & 2q  & \dots  & 3q-1   \\
x_2^{uv}      &     & 2q   & \dots  & 3q-r-1 & \dots & 3q-1  & 3q  & \dots  & 4q-1   \\
\vdots        &     &      & \vdots &        & \vdots &       &     & \vdots &        \\
x_i^{uv}      &     & iq   & \dots  & (i+1)q-r-1 & \dots & (i+1)q-1  & (i+1)q  & \dots  & (i+2)q-1   \\
\vdots        &     &      & \vdots &        & \vdots &       &     & \vdots &        \\
x_{t-1}^{uv}  &     & r-2q & \dots  & -q-1   & \dots & r-q-1 & r-q & \dots  & r-1 \\
x_t^{uv}      &     &      &        & p-1    & \dots & r-1   & r   & \dots  & q+r-1  &  \dots & 2q-1 
\end{array}
$$
}
\caption{Assignment of lists to the internal vertices of the forbidding path.}
\label{fig:pathlist}
\end{figure}

\begin{figure}[htb]
$$
\begin{array}{c|*{5}{r}|*{6}{r}}
\mbox{Vertex} & \multicolumn{11}{c}{\mbox{Lists}} \\ \hline
x_0^{uv}   & 17 &  0 &  1 &  2 &  3 &  4 &  5 &  6 &  7 \\
x_1^{uv}   &    &  4 &  5 &  6 &  7 &  8 &  9 & 10 & 11 \\
x_2^{uv}   &    &  8 &  9 & 10 & 11 & 12 & 13 & 14 & 15 \\
x_3^{uv}   &    & 12 & 13 & 14 & 15 & 16 & 17 &  0 &  1 \\
x_4^{uv}   &    &    & 17 &  0 &  1 &  2 &  3 &  4 &  5 &  6  & 7
\end{array}
$$
\caption{Assignment of lists to the internal vertices of the forbidding path
for the $(p,q)=(18,4)$ case.}
\label{fig:examplelists}
\end{figure}

Before moving on, let us check that the paths $P_{uv}$ and $P_{vu}$ actually do the job that they are meant to do; namely, that they forbid $u$ and $v$ from both being mapped to $S_0$. 

\begin{prop}\label{prop:forpath}
Let $uv$ be an edge of $G$. If $\psi$ is a $(p,q)$-colouring of $P_{uv}\cup P_{vu}$ such that the internal vertices of each path are coloured from their lists, then $\psi(u)$ and $\psi(v)$ cannot both be contained in $S_0$. 
\end{prop}

\begin{proof}
Suppose to the contrary that $\psi(u),\psi(v)\in S_0$. In $G_{p,q}$, the only edges contained in $S_0$ are between vertices in $[0,r-1]$ and vertices in $[q,q+r-1]$. So, without loss of generality, we can assume that 
\begin{equation}\label{leftright}\psi(u)\in [0,r-1]\text{ and }\psi(v)\in [q,q+r-1].\end{equation}
However, the fact that $\psi(u)\in [0,r-1]$ implies that $\psi\left(x_0^{uv}\right)\in [q,2q-1]$ (Note that $r-1 < q-1$ and thus $(r-1,p-1)$ is not an edge of $G_{p,q}$.)  By our observations above, we see that only colours to the right of the vertical line in Figure~\ref{fig:pathlist} can be used to colour $P_{uv}$.  In particular $\psi\left(x_t^{uv}\right) \in [r,2q-1]$. This implies that $\psi(v) \not\in [q, q+r-1]$, which contradicts (\ref{leftright}) and completes the proof.
\end{proof}

\subsection{Defining \texorpdfstring{$\boldsymbol{\alpha}$}{alpha} and \texorpdfstring{$\boldsymbol{\beta}$}{beta} on the Forbidding Paths}

As we have already mentioned, $\alpha$ and $\beta$ are defined in such a way that their restrictions to $G$ are equal to $\gamma f$ and $\gamma g$, respectively, and $\alpha(y_i)=\beta(y_i)=i$ for all $i$. Next, we will describe the way in which we extend $\alpha$ and $\beta$ to the vertices of the forbidding paths. 

Recall that the vertices of $G$ each receive a colour
from $\{ 0, q+r, 2q+r, \dots, (k-1)q+r \}$ under $\alpha$.
Given an edge $uv$ of $G$, the internal vertices of $P_{uv}$ are coloured as follows.
If $\alpha(u) \neq 0$ and $\alpha(v) \neq 0$, then colour $x_i^{uv}$ with colour
$iq$ for $i=0, 1, 2, \dots, t-1$ and colour $x_t^{uv}$ with $0$. (These colours correspond to the left hand column in Figure~\ref{fig:pathlist} for $1 \leq i \leq t-1$.) 
The path $P^{vu}$ is similarly coloured
so that $\alpha(x_{i}^{vu}) = iq$ and $\alpha(x_t^{vu}) = 0$.  In particular, 
$\alpha(x_0^{uv})=\alpha(x_t^{vu})=0$ for all $v \in N(u)$.  
Note that $(t-1)q \equiv r-2q \pmod{p}$ which is compatible with $0$.  
On the other hand, if $\alpha(u) = 0$, then set $\alpha(x_i^{uv}) = (i+1)q$ for $i=0, 1, \dots, t$. 
For $P_{vu}$, set $\alpha(x_i^{vu})=iq$ for $i=0, 1, \dots, t-1$ and $\alpha(x_t^{vu})=q$.
Observe in this case, the internal vertices around $u$ have colour $q$, 
i.e. $\alpha(x_0^{uv})=\alpha(x_t^{vu})=q$ for all $v \in N(u)$,
and the internal vertices around $v$ have colour $r$ or $0$, 
i.e. $\alpha(x_t^{uv})=r$ and $\alpha(x_0^{vu})=0$ for all $v \in N(u)$.
 
We similarly extend $\beta$ using $g$.
A technical detail we will exploit below is that in extending $\alpha$ and $\beta$ the colour for $x_t^{uv}$ belongs to $\{0, r, q \}$.  These colours are compatible with all colours above and to the left in the row for $x_{t-1}^{uv}$ by Proposition~\ref{prop:downandright}.
We call these colourings of $P_{uv}$ the \emph{standard path colourings}.  At the start of each
reconfiguration step we assume the paths have a standard path colouring, and the end of each 
reconfiguration step we ensure the paths have a standard path colouring.  The standard path colourings
correspond to coloumns in Figure~\ref{fig:pathlist} with possible changes at $x_t^{uv}$.

By Proposition~\ref{prop:forpath}, given any sequence of $(p,q)$-colourings which reconfigures $\alpha$ to $\beta$, we can compose each of these colourings with $\gamma$ to obtain a sequence of $k$-colourings taking $f$ to $g$. This proves the following proposition.

\begin{prop}\label{prop:toG}
Suppose $(G',\alpha,\beta)$ is an instance of $(p,q)$-\textsc{Recolouring} obtained from
$(G,f,g)$, an instance of $k$-\textsc{Recolouring}, as described above.  If $\alpha$
reconfigures to $\beta$, then $f$ reconfigures to $g$.
\end{prop}

\subsection{Recolouring \texorpdfstring{$\boldsymbol{G'}$}{G'}}

To complete the reduction we need to prove that if $f$ reconfigures to $g$, then $\alpha$ reconfigures to $\beta$. Let $(h_i)_{i=1}^s$ be any reconfiguration sequence taking $f$ to $g$. We show that there is a sequence $(\eta_i)_{i=1}^{s}$ of $(p,q)$-colourings of $G'$ such that
\begin{itemize}
\item $\eta_1=\alpha$ and $\eta_{s}=\beta$,
\item the restriction of $\eta_i$ to $G$ is $\gamma h_i$ for $1\leq i\leq s$, and
\item $\eta_i$ reconfigures to $\eta_{i+1}$ for $1\leq i\leq s-1$. 
\end{itemize}
Clearly this will prove that $\alpha$ reconfigures to $\beta$ and complete the proof of Theorem~\ref{circularThm}. 

Before tackling the general case we illustrate our method with the example
in Figure~\ref{fig:examplerecolour}, in which $(p,q)=(18,4)$. In this case, the colourings $\eta_i$ map $V(G)$ to the colours $\{0,6,10,14\}$. Let $1\leq j\leq s-1$ be fixed and suppose that $(\eta_i)_{i=1}^j$ have been constructed to satisfy the conditions above; our goal is to construct the colouring $\eta_{j+1}$ and a reconfiguration sequence taking $\eta_j$ to $\eta_{j+1}$. 

By definition, $h_j$ and $h_{j+1}$ differ on at most one vertex, say $u\in V(G)$. Suppose, for example that $h_j(u)=0$ and $h_{j+1}(u)=1$; this requires us change the colour of $u$ from $0$ (its current colour under $\eta_j$) to $6$ (its desired colour under $\eta_{j+1}$) without changing the colour of any other vertex of $G$.  We remark that this is the most involved case, and corresponds to 
Case~(c) in Lemma~\ref{lem:recolourpath} below.  The other cases use similar ideas but are more straightforward.  For each 
 vertex $v\in V(G)$ to which $u$ is adjacent, $v$ must receive a colour from $\{2,3\}$ under $h_j$ to facilitate the recolouring of $u$, which implies that $\eta_j(v)\in \{ 10, 14 \}$.  Each step below is applied for each neighbour $v$ of $u$, one by one, before moving on to the subsequent step.
The recoloured vertices at each step are shown as white vertices in Figure~\ref{fig:examplerecolour} with the new colours indicated
in boldface.  
\begin{list}{(\roman{enumi})}{\usecounter{enumi}}
 \item First, recolour $x_3^{vu}$ to $13$ and $x_4^{vu}$ to $7$. 
  Then, recolour $P_{uv}$ by  recolouring $x_4^{uv}$ to $5$, $x_3^{uv}$ to $1$, $x_2^{uv}$ to $15$, $x_1^{uv}$ to $11$ and $x_0^{uv}$ to $7$.  At this point all vertices of $\{x_0^{uv},x_t^{vu}: v\in N_G(u)\}$ have colour $7$.
 \item Now, change the colour of $u$ to $3$ (temporarily).
 \item Next, change the colour of $x_0^{uv}$ to $17$ and $x_t^{vu}$ to $17$.
  At this point all vertices of  $\{x_0^{uv},x_t^{vu}: v\in N_G(u)\}$ have colour $17$. 
  Recolour $u$ to $6$. 
 \item Recolour the paths to the standard colourings: $x_i^{uv}$ is coloured $iq$ for $i=0,1, \dots, t-1$
 and $x_t^{uv}$ is coloured $0$; $x_{t-1}^{vu}$ is coloured $12$ and $x_{t}^{vu}$ is coloured $0$. 
\end{list}

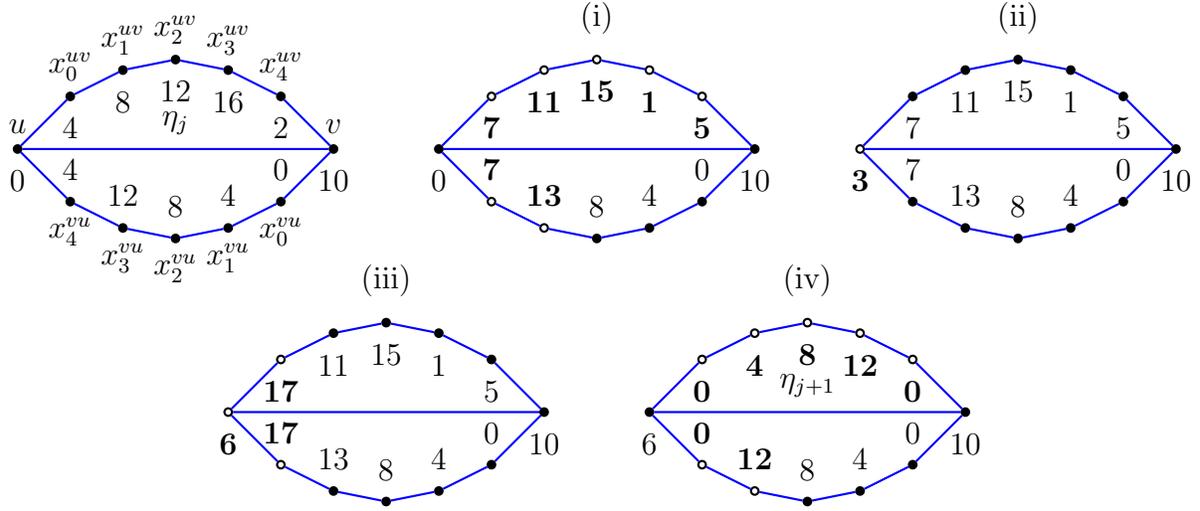
\begin{figure}
\vspace{-0.5em}
\begin{center}
\begin{tikzpicture}[scale=0.7]

\node[smallblack,label={$u$}] (u) at (0,1) {};
\node[smallblack,label={$x_0^{uv}$}] (Au) at (1,2) {};
\node[smallblack,label=below:{$x_4^{vu}$}] (Bu) at (1,0) {};
\node[smallblack,label={$x_1^{uv}$}] (Aup) at (2,2.5) {};
\node[smallblack,label=below:{$x_3^{vu}$}] (Bup) at (2,-0.5) {};
\node[smallblack,label={$x_2^{uv}$}] (C1) at (3,2.7) {};
\node[smallblack,label=below:{$x_2^{vu}$}] (C2) at (3,-0.7) {};
\node[smallblack,label={$x_3^{uv}$}] (Bvp) at (4,2.5) {};
\node[smallblack,label=below:{$x_1^{vu}$}] (Avp) at (4,-0.5) {};
\node[smallblack,label={$x_4^{uv}$}] (Bv) at (5,2) {};
\node[smallblack,label=below:{$x_0^{vu}$}] (Av) at (5,0) {};
\node[smallblack,label={$v$}] (v) at (6,1) {};

\draw[thick,blue] (v)--(Av)--(Avp)--(C2)--(Bup)--(Bu)--(u);
\draw[thick,blue] (v)--(Bv)--(Bvp)--(C1)--(Aup)--(Au)--(u);
\draw[thick,blue] (u)--(v);

\node[label=below:{$0$}] at (u) {};
\node[label=below:{$4$}] at (Au) {};
\node[label=below:{$8$}] at (Aup) {};
\node[label=below:{$12$}] at (C1) {};
\node[label=below:{$16$}] at (Bvp) {};
\node[label=below:{$2$}] at (Bv) {};
\node[label=below:{$10$}] at (v) {};

\node[label={$4$}] at (Bu) {};
\node[label={$12$}] at (Bup) {};
\node[label={$8$}] at (C2) {};
\node[label={$4$}] at (Avp) {};
\node[label={$0$}] at (Av) {};

\node at (3,1.5) {$\eta_j$};

\begin{scope}[xshift=8cm]
\node[smallblack] (u) at (0,1) {};
\node[smallwhite] (Au) at (1,2) {};
\node[smallwhite] (Bu) at (1,0) {};
\node[smallwhite] (Aup) at (2,2.5) {};
\node[smallwhite] (Bup) at (2,-0.5) {};
\node[smallwhite] (C1) at (3,2.7) {};
\node[smallblack] (C2) at (3,-0.7) {};
\node[smallwhite] (Bvp) at (4,2.5) {};
\node[smallblack] (Avp) at (4,-0.5) {};
\node[smallwhite] (Bv) at (5,2) {};
\node[smallblack] (Av) at (5,0) {};
\node[smallblack] (v) at (6,1) {};

\draw[thick,blue] (v)--(Av)--(Avp)--(C2)--(Bup)--(Bu)--(u);
\draw[thick,blue] (v)--(Bv)--(Bvp)--(C1)--(Aup)--(Au)--(u);
\draw[thick,blue] (u)--(v);

\node[label=below:{$0$}] at (u) {};
\node[label=below:{$\mathbf{7}$}] at (Au) {};
\node[label=below:{$\mathbf{11}$}] at (Aup) {};
\node[label=below:{$\mathbf{15}$}] at (C1) {};
\node[label=below:{$\mathbf{1}$}] at (Bvp) {};
\node[label=below:{$\mathbf{5}$}] at (Bv) {};
\node[label=below:{$10$}] at (v) {};

\node[label={$\mathbf{7}$}] at (Bu) {};
\node[label={$\mathbf{13}$}] at (Bup) {};
\node[label={$8$}] at (C2) {};
\node[label={$4$}] at (Avp) {};
\node[label={$0$}] at (Av) {};

\node at (3,1.5) {};
\node at (3,3.5) {(i)};
\end{scope}

\begin{scope}[xshift=16cm]
\node[smallwhite] (u) at (0,1) {};
\node[smallblack] (Au) at (1,2) {};
\node[smallblack] (Bu) at (1,0) {};
\node[smallblack] (Aup) at (2,2.5) {};
\node[smallblack] (Bup) at (2,-0.5) {};
\node[smallblack] (C1) at (3,2.7) {};
\node[smallblack] (C2) at (3,-0.7) {};
\node[smallblack] (Bvp) at (4,2.5) {};
\node[smallblack] (Avp) at (4,-0.5) {};
\node[smallblack] (Bv) at (5,2) {};
\node[smallblack] (Av) at (5,0) {};
\node[smallblack] (v) at (6,1) {};

\draw[thick,blue] (v)--(Av)--(Avp)--(C2)--(Bup)--(Bu)--(u);
\draw[thick,blue] (v)--(Bv)--(Bvp)--(C1)--(Aup)--(Au)--(u);
\draw[thick,blue] (u)--(v);

\node[label=below:{$\mathbf{3}$}] at (u) {};
\node[label=below:{$7$}] at (Au) {};
\node[label=below:{$11$}] at (Aup) {};
\node[label=below:{$15$}] at (C1) {};
\node[label=below:{$1$}] at (Bvp) {};
\node[label=below:{$5$}] at (Bv) {};
\node[label=below:{$10$}] at (v) {};

\node[label={$7$}] at (Bu) {};
\node[label={$13$}] at (Bup) {};
\node[label={$8$}] at (C2) {};
\node[label={$4$}] at (Avp) {};
\node[label={$0$}] at (Av) {};

\node at (3,1.5) {};
\node at (3,3.5) {(ii)};
\end{scope}

\begin{scope}[xshift=4cm, yshift=-5cm]
\node[smallwhite] (u) at (0,1) {};
\node[smallwhite] (Au) at (1,2) {};
\node[smallwhite] (Bu) at (1,0) {};
\node[smallblack] (Aup) at (2,2.5) {};
\node[smallblack] (Bup) at (2,-0.5) {};
\node[smallblack] (C1) at (3,2.7) {};
\node[smallblack] (C2) at (3,-0.7) {};
\node[smallblack] (Bvp) at (4,2.5) {};
\node[smallblack] (Avp) at (4,-0.5) {};
\node[smallblack] (Bv) at (5,2) {};
\node[smallblack] (Av) at (5,0) {};
\node[smallblack] (v) at (6,1) {};

\draw[thick,blue] (v)--(Av)--(Avp)--(C2)--(Bup)--(Bu)--(u);
\draw[thick,blue] (v)--(Bv)--(Bvp)--(C1)--(Aup)--(Au)--(u);
\draw[thick,blue] (u)--(v);

\node[label=below:{$\mathbf{6}$}] at (u) {};
\node[label=below:{$\mathbf{17}$}] at (Au) {};
\node[label=below:{$11$}] at (Aup) {};
\node[label=below:{$15$}] at (C1) {};
\node[label=below:{$1$}] at (Bvp) {};
\node[label=below:{$5$}] at (Bv) {};
\node[label=below:{$10$}] at (v) {};

\node[label={$\mathbf{17}$}] at (Bu) {};
\node[label={${13}$}] at (Bup) {};
\node[label={${8}$}] at (C2) {};
\node[label={${4}$}] at (Avp) {};
\node[label={${0}$}] at (Av) {};

\node at (3,1.5) {};
\node at (3,3.5) {(iii)};
\end{scope}

\begin{scope}[xshift=12cm, yshift=-5cm]
\node[smallblack] (u) at (0,1) {};
\node[smallwhite] (Au) at (1,2) {};
\node[smallwhite] (Bu) at (1,0) {};
\node[smallwhite] (Aup) at (2,2.5) {};
\node[smallwhite] (Bup) at (2,-0.5) {};
\node[smallwhite] (C1) at (3,2.7) {};
\node[smallblack] (C2) at (3,-0.7) {};
\node[smallwhite] (Bvp) at (4,2.5) {};
\node[smallblack] (Avp) at (4,-0.5) {};
\node[smallwhite] (Bv) at (5,2) {};
\node[smallblack] (Av) at (5,0) {};
\node[smallblack] (v) at (6,1) {};

\draw[thick,blue] (v)--(Av)--(Avp)--(C2)--(Bup)--(Bu)--(u);
\draw[thick,blue] (v)--(Bv)--(Bvp)--(C1)--(Aup)--(Au)--(u);
\draw[thick,blue] (u)--(v);

\node[label=below:{$6$}] at (u) {};
\node[label=below:{$\mathbf{0}$}] at (Au) {};
\node[label=below:{$\mathbf{4}$}] at (Aup) {};
\node[label=below:{$\mathbf{8}$}] at (C1) {};
\node[label=below:{$\mathbf{12}$}] at (Bvp) {};
\node[label=below:{$\mathbf{0}$}] at (Bv) {};
\node[label=below:{$10$}] at (v) {};

\node[label={$\mathbf{0}$}] at (Bu) {};
\node[label={$\mathbf{12}$}] at (Bup) {};
\node[label={$8$}] at (C2) {};
\node[label={$4$}] at (Avp) {};
\node[label={$0$}] at (Av) {};

\node at (3,1.5) {$\eta_{j+1}$};
\node at (3,3.5) {(iv)};
\end{scope}
\end{tikzpicture}
\end{center}
\vspace{-0.5em}

\caption{Recolouring $u$ from $0$ to $6$ in the $(18,4)$ case.}
\label{fig:examplerecolour}
\end{figure}

We now describe the recolouring technique in general.

\begin{lemma}\label{lem:recolourpath}
Let $h$ and $h'$ be $k$-colourings of $G$ which differ on a unique vertex $u$ and let $\eta$ be a $(p,q)$-colouring of $G'$ such that 
\begin{itemize}
\item $\eta$ is the standard path colouring for all paths $P_{uv}$,
\item $\eta(y_i)=i$ for all $i$, and 
\item the restriction of $\eta$ to $G$ is $\gamma h$. 
\end{itemize} 
Then there exists a $(p,q)$-colouring $\eta'$ of $G'$ such that $\eta'$ is the standard path colouring for all paths $P_{uv}$, the restriction of $\eta'$ to $G$ is $\gamma h'$, and $\eta$ reconfigures to $\eta'$.
\end{lemma}

\begin{proof}
We provide a reconfiguration sequence which takes $\eta$ to a colouring $\eta'$ with the desired properties. The proof is divided into cases. In each case, each of the reconfiguration steps is done for every neighbour $v$ of $u$, one by one, before moving on to the subsequent step.

As mentioned above, when extending $\alpha$ from $V(G)$ to $V(G')$, the colour for $x_t^{uv}$ is selected from the interval $[p-1,q+r]$.  In some cases below we may assign $x_t^{uv}$ a colour outside of this interval (but still on its list).  However, the assignment will be compatible with the current colour of $x_{t-1}^{uv}$ and at the end of the case the colour of $x_t^{uv}$ will be in $[p-1,q+r]$.

Let $uv$ be an edge in $G$ and assume the paths $P_{uv}$ and $P_{vu}$ have standard path colourings.

\begin{description}
  \item[Case (a): $\eta(u), \eta'(u) \in \{q+r, 2q+r, \dots, (k-1)q+r \}$]~
   In standard path colourings $\eta(x_0^{uv}), \eta(x_t^{vu}) \in \{ 0, r \}$ for each
   $v \in N(u)$.  The colour $\eta'(u)$ is compatible with both $0$ and $r$, so we may
   change the colour of $u$ to $\eta'(u)$.  The paths still have standard path colourings.

  \item[Case (b): $\eta(u)=0, \eta'(u) \in \{ 2q+r, 3q+r, \dots, (k-1)q+r \}$]~
  \begin{list}{(\roman{enumii})}{\usecounter{enumii}}
    \item By the definition of standard path colourings, $\eta(x_0^{uv})=\eta(x_t^{vu})=q$. Change
    the colour of $u$ to $\eta'(u)$.
    \item Recolour $x_t^{vu}$ to $0$.  The path $P_{vu}$ now has a standard path colouring.
    \item Recolour $x_i^{uv}$ to $iq$ for $i=0, 1, \dots, t-1$ and $x_t^{uv}$ to $0$.  The path
    $P_{uv}$ now has a standard path colouring.
  \end{list}

  \item[Case (c): $\eta(u)=0, \eta'(u) = q+r$]~
  \begin{list}{(\roman{enumii})}{\usecounter{enumii}}
    \item We begin observing that $h(u)=0$ and $h'(u)=1$, which implies that $h(v)\in\{2,\dots,k-1\}$. Thus, $\eta(v) \in \{2q+r, 3q+r, \dots, (k-1)q+r \}$. For $i=t,\dots,0$, in order, recolour $x_i^{uv}$ to $(i+2)q-1$.
    \item Change the colours of 
    $x_{t-1}^{vu}$ to $-q-1$, and $x_t^{vu}$ to $2q-1$. These colours are compatible with each other, and by Propsoition~\ref{prop:downandright}, the colour of $x_{t-1}^{vu}$ and $x_{t-2}^{vu}$ are compatible as well. 
    \item At this point all vertices of $\left\{x_0^{uv}, x_t^{vu}: v\in N_G(u)\right\}$  have colour $2q-1$.
    Recolour $u$ to $q-1$ (temporarily). Note that this colour is compatible with all neighbours $v$ of $u$ in $G$ since $\eta(v) \in \{ 2q+r, 3q+r, \dots, (k-1)q+r \}$.
    \item Recolour $x_0^{uv}$ to $p-1$ and $x_t^{vu}$ to $p-1$.
    \item At this point all vertices of $\left\{x_0^{uv}, x_t^{vu}: v\in N_G(u)\right\}$  are coloured $p-1$. Change the colour of $u$ to $q+r$.
    \item Recolour $x_{t-1}^{vu}$ to $r-2q$ and $x_t^{vu}$ to $0$.  Recolour $x_i^{uv}$ to $iq$ for
    $i=0, 1, \dots, t-1$ and $x_t^{uv}$ to $0$.  The paths now have the standard path colourings.
  \end{list}
    
  \item[Case (d): $\eta(u) \in \{ 2q+r, 3q+r, \dots (k-1)q+r \}, \eta'(u) = 0$]~
  \begin{list}{(\roman{enumii})}{\usecounter{enumii}}
    \item Change the colour of $x_t^{vu}$ to $q$, and recolour $x_i^{uv}$ to $(i+1)q$ for
    $i = t, t-1, \dots, 0$. 
    \item  At this point all vertices of $\left\{x_0^{uv}, x_t^{vu}: v\in N_G(u)\right\}$ have colour $q$.
    Now change the colour of $u$ to $0$. 
    \item The paths have standard path colourings.
  \end{list}  
  
  \item[Case (e): $\eta(u) = q+r, \eta'(u) = 0$]~
  \begin{list}{(\roman{enumii})}{\usecounter{enumii}}
    \item Recolour $x_t^{vu}$ to $p-1$ and $x_0^{uv}$ to $p-1$.
    \item At this point all vertices of $\left\{x_0^{uv}, x_t^{vu}: v\in N_G(u)\right\}$ have colour $p-1$.
    Change the colour of $u$ to $q-1$ (temporarily). Note that this colour is compatible with all neighbours $v$ of $u$ in $G$ since $\eta(v) \in \{ 2q+r, 3q+r, \dots, (k-1)q+r \}$.
    \item Since $\eta(v) \in \{2q+r, 3q+r, \dots, (k-1)q+r \}$, for $i=t,\dots,0$ we can recolour $x_i^{uv}$ to $(i+2)q-1$. 
    \item Change the colour of $x_{t-1}^{vu}$ to $-q-1$, $x_t^{vu}$ to $2q-1$ and $u$ to $0$.
    \item Finally, recolour $x_i^{uv}$ to $(i+1)q$ for $i=0, 1, \dots, t$. Recolour
    $x_t^{vu}$ to $q$, and $x_{t-1}^{vu}$ to $r-2q$.  The paths have 
    the standard path colourings.
  \end{list}
\end{description}
This completes the proof of the lemma, and of Theorem~\ref{circularThm}.
\end{proof}

\section{Cycles in Graphs of Large Chromatic Number}

A result of Chen and Saito~\cite{Chen1994} says that graphs without cycles of length divisible by three are $2$-degenerate. In particular, it follows that such graphs are $3$-colourable. Recently, Wrochna discovered a short and elegant proof of the $3$-colourability result using  ideas from~\cite{3col}. He has decided not to publish it himself, but has given us permission to include it here. 

\begin{thm}[Chen and Saito~\cite{Chen1994}]
\label{0mod3}
If $G$ contains no cycle of length $0\bmod3$, then $G$ is $3$-colourable.
\end{thm}

\begin{proof}[Proof (Wrochna)]
Suppose that the statement is false and let $G$ be a counterexample for which $|E(G)|$ is minimum. Define $G' := G-e$ where $e=uv$ is an edge of $G$.
By hypothesis, $G'$ admits a $3$-colouring, say $f:V(G')\to \{0,1,2\}$. If $f(u)\neq f(v)$, then we are done.  So we assume, without loss of generality, that $f(u)=f(v)=0$. 

Now, let $g:V(G)\to \{0,1,2\}$ be defined by $g(x)= f(x)+1\bmod3$ for all $x\in V(G)$. It is clear that
$\varphi_f(C)=\varphi_g(C)$ for all cycles $C$ in $G'$. Also, since $G'$ has no cycles of length $0\bmod3$, there are no directed cycles in $D_f$ or $D_g$, and therefore no fixed vertices either. Thus, by Theorem~\ref{thm:char}, $f$ can be reconfigured to $g$. Let $h$ be the first colouring of the reconfiguration sequence such that $h(u)\neq 0$ or $h(v)\neq 0$. Since the reconfiguration sequence colours only one vertex in each step, it is clear that $h$ is a $3$-colouring of $G$, which completes the proof.
\end{proof}

A closer look at the proof shows that it actually yields something slightly stronger: \emph{If $G$ contains an edge $e$ such that $G-e$ has no cycle of length $0\bmod 3$, then $\chi(G)\leq 3$.} Using a similar strategy, we can generalize this to $k$-colourings. 

\begin{thm}
\label{0modk}
If $G$ contains an edge $e$ such that $G-e$ contains fewer than $\frac{(k-1)!}{2}$ cycles of length $0\bmod k$, then $\chi(G)\leq k$. 
\end{thm}

\begin{proof}
We proceed by induction on $|E(G)|$, where the case $|E(G)|=0$ is trivial. Let $e=uv$ be an edge of $G$ such that $G':=G-e$ contains fewer than $\frac{(k-1)!}{2}$ cycles of length $0\bmod k$. Then, by the inductive hypothesis, $G'$ admits a $k$-colouring $f:V(G')\to \{0,\dots,k-1\}$. If $f(u)\neq f(v)$, then we are done. So we assume, without loss of generality, that $f(u)=f(v)=0$. 

Let $\pi$ be a permutation of $\{0,\dots,k-1\}$ with $\pi(0)=0$.  Let $F_{f,\pi}$ be an oriented spanning subgraph of $G'$ where $F_{f,\pi}$ contains an arc $\overrightarrow{xy}$ if  $f(x) = \pi(i)$ and $f(y) = \pi(i+1)$ for some $i$ (addition is modulo $k$). Let $S_f$ be the set of vertices in $G$ which can be reached by an oriented path in $F_{f,\pi}$ starting at $u$. Suppose that $F_{f,\pi}[S_f]$ does not contain a directed cycle. Then $F_{f,\pi}[S_f]$ contains a sink $x$. Recolour $x$ from say $\pi(j)$ to $\pi(j+1)$, i.e. recolour $x$ to $\pi(\pi^{-1}(f(x))+1)$.  This results in a proper $k$-colouring $f'$ of $G'$ in which $F_{f',\pi}[S_{f'}]=F_{f,\pi}[S_f]- x$. Repeating this procedure, one eventually reaches a situation where either $u$ or $v$ is a sink and so the colour of either $u$ or $v$ is changed; hence, there is a first $k$-colouring $g$ where $g(u)\neq g(v)$ and we are done.

Therefore, $F_{f,\pi}$ must contain a directed cycle for every permutation $\pi$ with $\pi(0)=0$. Such a cycle must have vertex colours $\pi(i), \pi(i+1), \pi(i+2), \dots, \pi(i-1), \pi(i)$, and thus have length $0 \bmod{k}$.  On the other hand, let $C$ be a cycle of length $0 \bmod{k}$ in $G'$.  The colours of $C$ under $f$ have the following property: there is a vertex of colour $0$, starting at that vertex and traversing the cycle we see all $k$ colours in the first $k$ vertices, and each successive block of $k$ vertices is coloured with the same colours in the same order.  In other words, the first $k$ vertices define a permutation $\pi$ such that $\pi(0) = 0$, and $C$ is a directed cycle in $F_{f,\pi}$.  Observe if we traverse $C$ in the opposite direction we obtain a second permutation $\pi'$ where $C$ is directed in $F_{f,\pi'}$.  Moreover if $C$ is a directed cycle in $F_{f,\sigma}$ with $\sigma(0) = 0$, then $\sigma$ must be $\pi$ or $\pi'$. Since $G'$ has fewer than $(k-1)!/2$ cycles of length $0 \bmod {k}$, there must be some permutation $\pi$ such that $F_{f,\pi}$ is acyclic.  The result follows.
\end{proof}

\begin{cor}
\label{0modkcor}
If $\chi(G)>k$, then $G$ contains at least $\frac{(k-1)!}{2}$ cycles of length $0\bmod k$.
\end{cor}

The complete graph of order $k+1$ has precisely $\frac{(k+1)(k-1)!}{2}$ cycles of length $0\bmod k$, and so Corollary~\ref{0modkcor} is within a factor $k+1$ of being tight. We wonder whether $K_{k+1}$ contains the fewest cycles of length $0\bmod k$ among all non-$k$-colourable graphs. 

\begin{conj}
\label{cycles}
If $\chi(G)>k$, then $G$ contains at least $\frac{(k+1)(k-1)!}{2}$ cycles of length $0\bmod k$.
\end{conj}

The case $k=3$ may be particularly instructive: \emph{Is it true that every graph $G$ with $\chi(G)\geq 4$ contains at least $4$ cycles of length $0\bmod3$?} If Conjecture~\ref{cycles} turns out to be false, then it would still be interesting to determine the minimum number of cycles of length $0\bmod k$ in a graph of chromatic number greater than $k$. 

More generally, one could investigate the minimum number cycles of length $r\bmod k$ in a graph of chromatic number at least, say, $f(r,k)$. With regards to the \emph{existence} of such cycles, Chen, Ma and Zang~\cite{Ma} proved that any graph with chromatic number greater than $k$ must contain a cycle of length $r\bmod k$ for $r\in\{0,\dots,k-1\}\setminus\{2\}$ and that any graph of chromatic number greater than $k+1$ must also contain a cycle of length $2\bmod k$. Dean, Lesniak and Saito~\cite{Dean} proved that if $\chi(G)\geq4$, then $G$ has a cycle of length $0\bmod 4$. Similar problems for \emph{induced} cycles are very well studied but usually, in this setting, the function $f$ also depends on the size of the largest clique in $G$; see, e.g.~\cite{Bonamy,Holes2,Holes3,Gyarfas,Lagoutte,Holes1,Holes4}.

\begin{ack}
The authors would like to thank Marcin Wrochna for allowing us to share his lovely proof of Theorem~\ref{0mod3}. We would also like to thank Jie Ma for directing our attention to~\cite{Chen1994} and~\cite{Ma}. The fourth author would like to thank Marthe Bonamy and Guillem Perarnau for stimulating discussions about reconfiguration problems during a workshop at the Bellairs Institute of McGill University in Holetown, Barbados in 2015. Finally, we wish to thank the anonymous referees for their many helpful suggestions.
\end{ack}


\begin{thebibliography}{23}
\providecommand{\natexlab}[1]{#1}
\providecommand{\url}[1]{\texttt{#1}}
\expandafter\ifx\csname urlstyle\endcsname\relax
  \providecommand{\doi}[1]{doi: #1}\else
  \providecommand{\doi}{doi: \begingroup \urlstyle{rm}\Url}\fi

\bibitem{AhoHU83}
A.~V. Aho, J.~E. Hopcroft, and J.~D. Ullman.
\newblock \emph{Data Structures and Algorithms}.
\newblock Addison-Wesley, 1983.

\bibitem{Bonamy}
M.~Bonamy, P.~Charbit, and S.~Thomass\'{e}.
\newblock Graphs with large chromatic number induce $3k$-cycles.
\newblock Submitted for publication, arXiv:1408.2172v1, August 2014.

\bibitem{Bonsma}
P.~Bonsma and L.~Cereceda.
\newblock Finding paths between graph colourings: {PSPACE}-completeness and
  superpolynomial distances.
\newblock \emph{Theoret. Comput. Sci.}, 410\penalty0 (50):\penalty0 5215--5226,
  2009.

\bibitem{BrewsterNoel}
R.~C. Brewster and J.~A. Noel.
\newblock Mixing homomorphisms, recolorings, and extending circular
  precolorings.
\newblock \emph{J. Graph Theory}, 80\penalty0 (3):\penalty0 173--198, 2015.

\bibitem{followUp}
R.~C. Brewster, S.~McGuinness, B.~Moore, and J.~A. Noel.
\newblock On the complexity of the reconfiguration problem for graph
  homomorphisms.
\newblock In preparation.

\bibitem{CerecedaThesis}
L.~Cereceda.
\newblock \emph{Mixing Graph Colourings}.
\newblock PhD thesis, The London School of Economics and Political Science,
  2007.

\bibitem{3col}
L.~Cereceda, J.~van~den Heuvel, and M.~Johnson.
\newblock Finding paths between 3-colorings.
\newblock \emph{J. Graph Theory}, 67\penalty0 (1):\penalty0 69--82, 2011.


\bibitem{Chen1994}
G.~Chen and A.~Saito.
\newblock Graphs with a cycle of length divisible by three.
\newblock \emph{J. Combin. Theory Ser. B}, 60\penalty0 (2):\penalty0 277--292,
  1994.

\bibitem{Ma}
Z.~Chen, J.~Ma, and W.~Zang.
\newblock Coloring digraphs with forbidden cycles.
\newblock \emph{J. Combin. Theory Ser. B}, 115\penalty0:\penalty0 210--223,
  2015.


\bibitem{Holes3}
M.~Chudnovsky, A.~Scott, and P.~Seymour.
\newblock Induced subgraphs of graphs with large chromatic number. {III}.
  {L}ong holes.
\newblock Accepted for publication, arXiv:1506.02232v2, March 2015.

\bibitem{Holes2}
M.~Chudnovsky, A.~Scott, and P.~Seymour.
\newblock Induced subgraphs of graphs with large chromatic number. {II}.
  {T}hree steps towards {G}y\'arf\'as' conjectures.
\newblock \emph{J. Combin. Theory Ser. B}, 118:\penalty0 109--128, 2016.

\bibitem{Dean}
N.~Dean, L.~Lesniak, and A.~Saito.
\newblock Cycles of length $0$ modulo $4$ in graphs.
\newblock \emph{Discrete Math.}, 121\penalty0 (1-3):\penalty0 37--49, 1993.

\bibitem{Gyarfas}
A.~Gy{\'a}rf{\'a}s.
\newblock Problems from the world surrounding perfect graphs.
\newblock In \emph{Proceedings of the {I}nternational {C}onference on
  {C}ombinatorial {A}nalysis and its {A}pplications ({P}okrzywna, 1985)},
  volume~19, pages 413--441, 1987.
  
\bibitem{survey}
J.~van~den Heuvel.
\newblock The complexity of change.
\newblock In \emph{Surveys in combinatorics 2013}, volume 409 of \emph{London
  Math. Soc. Lecture Note Ser.}, pages 127--160. Cambridge Univ. Press,
  Cambridge, 2013.

\bibitem{Ito2011}
T.~Ito, E.~D. Demaine, N.~J.~A. Harvey, C.~H. Papadimitriou, M.~Sideri,
  R.~Uehara, and Y.~Uno.
\newblock On the complexity of reconfiguration problems.
\newblock \emph{Theoret. Comput. Sci.}, 412\penalty0 (12-14):\penalty0
  1054--1065, 2011.

\bibitem{Ito2014}
T.~Ito, K.~Kawamura, H.~Ono, and X.~Zhou.
\newblock Reconfiguration of list {$L(2,1)$}-labelings in a graph.
\newblock \emph{Theoret. Comput. Sci.}, 544:\penalty0 84--97, 2014.

\bibitem{Jakob}
R.~Jakob.
\newblock Standortplanung mit blick auf online-strategien.
\newblock Graduate thesis, Universität Würzburg, 1997.

\bibitem{Johnson2014}
M.~Johnson, D.~Kratsch, S.~Kratsch, V.~Patel, and D.~Paulusma.
\newblock Finding shortest paths between graph colourings.
\newblock In \emph{Parameterized and exact computation}, volume 8894 of
  \emph{Lecture Notes in Comput. Sci.}, pages 221--233. Springer, Cham, 2014.

\bibitem{Lagoutte}
A.~Lagoutte.
\newblock Coloring graphs with no even hole $\geq 6$: the triangle-free case.
\newblock Submitted for publication, arXiv:1503.08057v2, June 2015.

\bibitem{Holes1}
A.~Scott and P.~Seymour.
\newblock Induced subgraphs of graphs with large chromatic number. {I}. {O}dd
  holes.
\newblock In press, arXiv:1410.4118v3, August 2015{\natexlab{a}}.

\bibitem{Holes4}
A.~Scott and P.~Seymour.
\newblock Induced subgraphs of graphs with large chromatic number. {IV}.
  {Co}nsecutive holes.
\newblock Submitted for publication, arXiv:1509.06563v1, September
  2015{\natexlab{b}}.

\bibitem{Vince}
A.~Vince.
\newblock Star chromatic number.
\newblock \emph{J. Graph Theory}, 12\penalty0 (4):\penalty0 551--559, 1988.

\bibitem{Marcin}
M.~Wrochna.
\newblock Reconfiguration and structural graph theory.
\newblock Master's thesis, University of Warsaw, 2014.

\bibitem{C4free}
M.~Wrochna.
\newblock Homomorphism reconfiguration via homotopy.
\newblock In \emph{32nd {I}nternational {S}ymposium on {T}heoretical {A}spects
  of {C}omputer {S}cience}, volume~30 of \emph{LIPIcs. Leibniz Int. Proc.
  Inform.}, pages 730--742. Schloss Dagstuhl. Leibniz-Zent. Inform., Wadern,
  2015.

\bibitem{Zhusurvey}
X.~Zhu.
\newblock Circular chromatic number: a survey.
\newblock \emph{Discrete Math.}, 229\penalty0 (1-3):\penalty0 371--410, 2001.

\end{thebibliography}
\end{document}